\documentclass{amsart}
\usepackage[T1]{fontenc}   

\usepackage{prettyref}
\usepackage{braket}
\usepackage{mathtools}
\usepackage[all]{xy}

\newtheorem{thm}{\protect\theoremname}[section]
\newtheorem{prop}[thm]{\protect\propositionname}
\newtheorem{lem}[thm]{\protect\lemmaname}
\newtheorem{cor}[thm]{\protect\corollaryname}
\theoremstyle{definition}
\newtheorem{defn}[thm]{\protect\definitionname}
\newtheorem{example}[thm]{\protect\examplename}
\theoremstyle{remark}
\newtheorem{claim}[thm]{\protect\claimname}

\providecommand{\claimname}{Claim}
\providecommand{\corollaryname}{Corollary}
\providecommand{\definitionname}{Definition}
\providecommand{\examplename}{Example}
\providecommand{\lemmaname}{Lemma}
\providecommand{\propositionname}{Proposition}
\providecommand{\theoremname}{Theorem}

\newrefformat{prop}{Proposition \ref{#1}}
\newrefformat{cor}{Corollary \ref{#1}}
\newrefformat{claim}{Claim \ref{#1}}
\newrefformat{exa}{Example \ref{#1}}

\begin{document}

\noindent                                             
\begin{picture}(150,36)                               
\put(5,20){\tiny{Submitted to}}                       
\put(5,7){\textbf{Topology Proceedings}}              
\put(0,0){\framebox(140,34){}}                        
\put(2,2){\framebox(136,30){}}                        
\end{picture}                                        
\vspace{0.5in}

\renewcommand{\bf}{\bfseries}
\renewcommand{\sc}{\scshape}
\vspace{0.5in}

\newcommand{\ns}[1]{\prescript{\ast}{}{#1}}
\newcommand{\st}[1]{\prescript{\circ}{}{#1}}
\newcommand{\PNS}{\operatorname{PNS}}
\newcommand{\SPNS}{\operatorname{SPNS}}

\title{Asymmetric completions of partial metric spaces}

\author{Takuma Imamura}
\address{Research Institute for Mathematical Sciences; Kyoto University; Kitashirakawa Oiwake-cho, Sakyo-ku, Kyoto 606-8502, Japan}
\email{timamura@kurims.kyoto-u.ac.jp}



\begin{abstract}
Ge and Lin (2015) proved the existence and the uniqueness of p-Cauchy
completions of partial metric spaces under symmetric denseness. They
asked if every (non-empty) partial metric space $X$ has a p-Cauchy
completion $\bar{X}$ such that $X$ is dense but not symmetrically
dense in $\bar{X}$. We construct asymmetric p-Cauchy completions
for all non-empty partial metric spaces. This gives a positive answer
to the question. We also provide a nonstandard construction of partial
metric completions.
\end{abstract}

\maketitle

\section*{Introduction}

Metric spaces are among the most investigated types of spaces. Whilst
all metric spaces are $T_{1}$, non-$T_{1}$ spaces have also been
paid attention, particularly in the context of denotational semantics
of programming languages. In order to deal with such spaces in a similar
fashion to metric spaces, Matthews \cite{Mat92b,Mat94} introduced
the notion of partial metric. Roughly speaking, a partial metric space
is a generalised metric space where the self-distance is not necessarily
zero.
\begin{defn}[\cite{Mat94}]
A \emph{partial metric} (abbr. pmetric) on a set $X$ is a function
$p_{X}\colon X\times X\to\mathbb{R}_{\geq0}$ that satisfies the following
axioms:
\begin{enumerate}
\item[(P1)] $p_{X}\left(x,x\right)=p_{X}\left(x,y\right)=p_{X}\left(y,y\right)\implies x=y$;
\item[(P2)] $p_{X}\left(x,x\right)\leq p_{X}\left(x,y\right)$;
\item[(P3)] $p_{X}\left(x,y\right)=p_{X}\left(y,x\right)$;
\item[(P4)] $p_{X}\left(x,z\right)+p_{X}\left(y,y\right)\leq p_{X}\left(x,y\right)+p_{X}\left(y,z\right)$.
\end{enumerate}
A set equipped with a partial metric is called a \emph{partial metric
space} (abbr. pmetric space).
\end{defn}

Let $X$ be a pmetric space. For $x\in X$ and $\varepsilon\in\mathbb{R}_{+}$,
the open ball of centre $x$ and radius $\varepsilon$ is defined
by
\[
B_{\varepsilon}\left(x\right):=\set{y\in X|p_{X}\left(x,y\right)<p_{X}\left(x,x\right)+\varepsilon}.
\]
The family of all open balls generates a topology on $X$, called
the \emph{induced topology}. We always assume that $X$ is equipped
with the induced topology.
\begin{defn}[\cite{GL15}]
Let $X$ be a pmetric space and $A$ a subset of $X$.
\begin{enumerate}
\item $A$ is said to be \emph{(topologically) dense} in $X$ if for any
$x\in X$ and for any $\varepsilon>0$ there is a $y\in A$ such that
$y\in B_{\varepsilon}\left(x\right)$.
\item $A$ is said to be \emph{symmetrically dense} in $X$ if for any $x\in X$
and for any $\varepsilon>0$ there is a $y\in A$ such that $x\in B_{\varepsilon}\left(y\right)$
and $y\in B_{\varepsilon}\left(x\right)$.
\end{enumerate}
\end{defn}

Symmetric denseness is stronger than the usual (topological) denseness.
If $X$ is a metric space, then the conditions ``$x\in B_{\varepsilon}\left(y\right)$''
and ``$y\in B_{\varepsilon}\left(x\right)$'' are equivalent, so topological
and symmetric denseness coincide.
\begin{defn}[\cite{Mat94}]
Let $X$ be a pmetric and $\set{x_{n}}_{n\in\mathbb{N}}$ a sequence
in $X$.
\begin{enumerate}
\item $\set{x_{n}}_{n\in\mathbb{N}}$ is said to \emph{p-converge to} $x\in X$
if $p_{X}\left(x,x\right)=\lim_{n\to\infty}p_{X}\left(x,x_{n}\right)=\lim_{n\to\infty}p_{X}\left(x_{n},x_{n}\right)$.
\item $\set{x_{n}}_{n\in\mathbb{N}}$ is said to be \emph{p-Cauchy} if $\lim_{n,m\to\infty}p_{X}\left(x_{n},x_{m}\right)$
exists and is finite.
\end{enumerate}
The space $X$ is said to be \emph{p-Cauchy complete} if every p-Cauchy
sequence p-converges.
\end{defn}

\begin{defn}[\cite{Rom09}]
Let $X$ be a pmetric and $\set{x_{n}}_{n\in\mathbb{N}}$ a sequence
in $X$.
\begin{enumerate}
\item $\set{x_{n}}_{n\in\mathbb{N}}$ is said to $0$-converge to $x\in X$
if $p_{X}\left(x,x\right)=\lim_{n\to\infty}p_{X}\left(x,x_{n}\right)=\lim_{n\to\infty}p_{X}\left(x_{n},x_{n}\right)=0$.
\item $\set{x_{n}}_{n\in\mathbb{N}}$ is said to be $0$-Cauchy if $\lim_{n,m\to\infty}p_{X}\left(x_{n},x_{m}\right)=0$.
\end{enumerate}
The space $X$ is said to be \emph{$0$-Cauchy complete} if every
$0$-Cauchy sequence $0$-converges.
\end{defn}

$0$-convergence implies p-convergence; p-convergence implies the
usual (topological) convergence; and $0$-Cauciness implies p-Cauciness.
None of the implications is reversible in general. For $0$-Cauchy
sequences, p-convergence implies $0$-convergence. As a consequence,
p-Cauchy completeness implies $0$-Cauchy completeness. If $X$ is
a metric space, all of the implications are reversible. Some fixed-point
results on Cauchy complete metric spaces can be generalised to p-Cauchy
(or $0$-Cauchy) complete pmetric spaces \cite{IPR11,Mat92a,Mat92b,Mat94,Mat95,Rom09,Wad81}.
\begin{defn}[\cite{GL15}]
A \emph{p-Cauchy completion} of a pmetric space $X$ is a pair $\left(\bar{X},i\right)$
of a p-Cauchy complete pmetric space $\bar{X}$ and an isometric embedding
$i\colon X\to\bar{X}$ such that $i\left(X\right)$ is (topologically)
dense in $\bar{X}$.
\end{defn}

Ge and Lin \cite{GL15} proved the existence and the uniqueness of
p-Cauchy completions of pmetric spaces. Actually they proved that
every pmetric space $X$ has a unique (up to isometry) p-Cauchy completion
$\bar{X}$ in which $X$ is \emph{symmetrically} dense. It was open
whether the assumption of symmetric denseness in the uniqueness theorem
can be weakened to the usual denseness. It was also open whether every
(non-empty) pmetric space $X$ has a p-Cauchy completion $\bar{X}$
such that $X$ is dense but not symmetrically dense in $\bar{X}$.
Dung \cite{Dun17} provided \emph{an example of} p-Cauchy complete
pmetric spaces having dense but not symmetrically dense subsets, and
negatively answered the first problem.

The aim of \prettyref{sec:General-construction} is to solve the second
problem. We first give a quite simple example of a pmetric space that
has two different p-Cauchy completions. We then construct ``asymmetric''
p-Cauchy completions for \emph{all} non-empty pmetric spaces. So the
second problem is answered affirmatively. This also implies that symmetric
denseness in the uniqueness theorem cannot be weakened to topological
denseness in \emph{any} cases (except in the trivial case). Thus the
first problem is strongly negative. Because of this, symmetric denseness
might be more suitable for pmetric spaces than topological denseness.
In \prettyref{sec:Nonstandard-treatment}, we treat pmetric completions
\emph{via nonstandard analysis}.

We refer to \cite{BKMP09,HWZ17,Mat94} for basics on pmetrics and
\cite{HR85,Rob66,SL76} for nonstandard analysis. Nonstandard analysis
is used only within \prettyref{sec:Nonstandard-treatment}.

\section{\label{sec:General-construction}General construction of asymmetric
p-Cauchy completions}

Ge and Lin \cite{GL15} proved that every pmetric space $X$ has
a unique (up to isometry) p-Cauchy completion $\bar{X}$ where $X$
is symmetrically dense. The proof of the uniqueness theorem depends
on the assumption of symmetric denseness. In fact, there is a counterexample
when assuming just topological denseness.
\begin{prop}[\cite{Dun17}]
\label{prop:Two-completions}There exists a pmetric space that has
two p-Cauchy completions up to isometry.
\end{prop}

\begin{proof}
Let $X$ be a one-point space $\set{a}$ together with a pmetric 
\[
p_{X}\left(a,a\right):=0.
\]
Then $X$ is p-Cauchy complete. In other words, $X$ is a p-Cauchy
completion of itself.

On the other hand, let $Y$ be a two-point space $\set{a,b}$ together
with a pmetric
\[
p_{Y}\left(a,a\right):=0,\ p_{Y}\left(a,b\right)=p_{Y}\left(b,a\right)=p_{Y}\left(b,b\right):=1.
\]
Then $Y$ is also p-Cauchy complete. Obviously the inclusion map $i\colon X\to Y$
is an isometry. For each $y\in Y$ and $\varepsilon>0$, since $p_{Y}\left(y,a\right)\leq p_{Y}\left(y,y\right)<p_{Y}\left(y,y\right)+\varepsilon$,
we have $a\in B_{\varepsilon}\left(y\right)$. Hence $X$ is dense
in $Y$. Thus $Y$ is a p-Cauchy completion of $X$. The cardinality
of $Y$ is greater than that of $X$, so $Y$ is not isometric to
$X$.
\end{proof}
\begin{prop}[\cite{Dun17}]
\label{prop:Asymmetric-completion}There exist a pmetric space $X$
and its p-Cauchy completion $\tilde{X}$ such that $X$ is not symmetrically
dense in $\tilde{X}$.
\end{prop}

\begin{proof}
Let $X$ and $Y$ be as above. Since $p_{Y}\left(a,b\right)\geq p_{Y}\left(a,a\right)+1$,
it follows that $b\notin B_{\varepsilon}\left(a\right)$ for any $0<\varepsilon<1$.
Thus $b$ witnesses that $X$ is not symmetrically dense in $Y$.
\end{proof}
In the proof of \prettyref{prop:Asymmetric-completion} the asymmetric
p-Cauchy completion $Y$ of $X$ is obtained by attaching an extra
\emph{asymmetric} point $b$. Asymmetric p-Cauchy completions of general
pmetric spaces can also be obtained in a similar way.
\begin{lem}
\label{lem:Asymmetric-superset}For every non-empty pmetric space
$X$, there exists a pmetric space $Y\supseteq X$ such that $X$
is dense but not symmetrically dense in $Y$.
\end{lem}

\begin{proof}
Fix a base point $a\in X$. Consider a new point (say $b$) and let
$Y:=X\cup\set{b}$. Define a function $p_{Y}\colon Y\times Y\to\mathbb{R}$
by
\[
p_{Y}\left(x,y\right):=\begin{cases}
p_{X}\left(x,y\right), & x,y\in X,\\
p_{X}\left(a,y\right)+1, & x=b,\ y\in X,\\
p_{X}\left(x,a\right)+1, & x\in X,\ y=b,\\
p_{X}\left(a,a\right)+1, & x=y=b.
\end{cases}
\]
Clearly $p_{Y}$ extends $p_{X}$. For every $\varepsilon>0$, since
$p_{Y}\left(b,a\right)=p_{Y}\left(b,b\right)<p_{Y}\left(b,b\right)+\varepsilon$,
we have $a\in B_{\varepsilon}\left(b\right)$. Hence $X$ is dense
in $Y$. On the other hand, for any $x\in X$, since $p_{Y}\left(x,b\right)=p_{X}\left(x,a\right)+1\geq p_{Y}\left(x,x\right)+1$,
it follows that $b\notin B_{\varepsilon}\left(x\right)$ for any $0<\varepsilon<1$.
Hence $X$ is not symmetrically dense in $Y$.

We only need to verify that $p_{Y}$ is a pmetric.
\begin{enumerate}
\item[(P1)] Let $x,y\in Y$ and suppose that $x\neq y$. Case 1: $x,y\in X$.
Two of $p_{Y}\left(x,x\right)$, $p_{Y}\left(x,y\right)$ and $p_{Y}\left(y,y\right)$
are different by (P1) for $p_{X}$. Case 2: $x=b$ and $y\in X$.
Then $p_{Y}\left(y,y\right)=p_{X}\left(y,y\right)\leq p_{X}\left(a,y\right)<p_{X}\left(a,y\right)+1=p_{Y}\left(b,y\right)$
by (P2) and (P3) for $p_{X}$. Hence $p_{Y}\left(y,y\right)\neq p_{Y}\left(x,y\right)$.
Case 3: $x\in X$ and $y=b$. Similarly to Case 2, we have that $p_{Y}\left(x,x\right)\neq p_{Y}\left(x,y\right)$.
In each case, two of $p_{Y}\left(x,x\right)$, $p_{Y}\left(x,y\right)$
and $p_{Y}\left(y,y\right)$ are different.
\item[(P2)] Let $x,y\in Y$. Case 1: $x,y\in X$. Immediate from (P2) for $p_{X}$.
Case 2: $x=b$ and $y\in X$. By (P2) for $p_{X}$, we have that $p_{Y}\left(b,b\right)=p_{X}\left(a,a\right)+1\leq p_{X}\left(a,y\right)+1\leq p_{Y}\left(b,y\right)$.
Hence $p_{Y}\left(x,x\right)\leq p_{Y}\left(x,y\right)$. Case 3:
$x\in X$ and $y=b$. Similarly to the previous case. Case 4: $x=y=b$.
This case is trivial.
\item[(P3)] The definition of $p_{Y}$ is symmetric with respect to $x$ and
$y$. So (P3) for $p_{X}$ implies (P3) for $p_{Y}$.
\item[(P4)] Let $x,y,z\in Y$. If $x=y$, then $p_{Y}\left(x,z\right)+p_{Y}\left(y,y\right)=p_{Y}\left(y,z\right)+p_{Y}\left(x,y\right)$.
If $y=z$, then $p_{Y}\left(x,z\right)+p_{Y}\left(y,y\right)=p_{Y}\left(x,y\right)+p_{Y}\left(y,z\right)$.
If $x=z$, then by (P2) and (P3) for $p_{Y}$, we have $p_{Y}\left(x,z\right)+p_{Y}\left(y,y\right)=p_{Y}\left(x,x\right)+p_{Y}\left(y,y\right)\leq p_{Y}\left(x,y\right)+p_{Y}\left(y,z\right)$.
Thus we may assume that all of $x$, $y$ and $z$ are different.
Case 1: $x,y,z\in X$. Immediate from (P4) for $p_{X}$. Case 2: $x=b$
and $y,z\in X$. Then, by (P4) for $p_{Y}$,
\begin{align*}
p_{Y}\left(x,z\right)+p_{Y}\left(y,y\right) & =p_{Y}\left(b,z\right)+p_{Y}\left(y,y\right)\\
 & =p_{X}\left(a,z\right)+1+p_{X}\left(y,y\right)\\
 & \leq p_{X}\left(a,y\right)+p_{X}\left(y,z\right)+1\\
 & =p_{Y}\left(b,y\right)+p_{Y}\left(y,z\right)\\
 & =p_{Y}\left(x,y\right)+p_{Y}\left(y,z\right).
\end{align*}
Case 3: $x\in X$, $y=b$ and $z\in X$. By the definition of $p_{X}$
and (P4) for $p_{X}$,
\begin{align*}
p_{Y}\left(x,z\right)+p_{Y}\left(y,y\right) & =p_{X}\left(x,z\right)+p_{X}\left(a,a\right)+1\\
 & \leq p_{X}\left(x,a\right)+p_{X}\left(a,z\right)+1\\
 & =\left(p_{Y}\left(x,b\right)-1\right)+\left(p_{Y}\left(b,z\right)-1\right)+1\\
 & \leq p_{Y}\left(x,b\right)+p_{Y}\left(b,z\right)\\
 & =p_{Y}\left(x,y\right)+p_{Y}\left(y,z\right).
\end{align*}
Case 4: $x,y\in X$ and $z=b$. Similarly to Case 2.\qedhere
\end{enumerate}
\end{proof}
\begin{thm}
\label{thm:Asymmetric-completions}For every non-empty pmetric space
$X$, there exists a p-Cauchy completion $\tilde{X}$ of $X$ such
that $X$ is not symmetrically dense in $\tilde{X}$.
\end{thm}

\begin{proof}
By \prettyref{lem:Asymmetric-superset} we can construct a pmetric
space $Y\supseteq X$ such that $X$ is dense but not symmetrically
dense in $Y$. By \cite[Theorem 2]{GL15} $Y$ has a p-Cauchy completion
$\bar{Y}\supseteq Y$. Then $\bar{Y}$ is a p-Cauchy completion of
$X$ in which $X$ is dense but not symmetrically dense.
\end{proof}
Symmetric denseness is invariant under isometry, i.e. if there is
an isometry $f\colon\left(X,A\right)\to\left(Y,B\right)$, then $A$
is symmetrically dense in $X$ if and only if $B$ is symmetrically
dense in $Y$. From this fact, we obtain the following corollary.
\begin{cor}
Every non-empty pmetric space has at least two p-Cauchy completions
up to isometry.
\end{cor}

\begin{example}
\label{exa:Kahn-domain}A more natural and non-trivial example comes
from the \emph{Kahn domain }$\Sigma^{\leq\omega}$. Let $\Sigma$
be a (non-empty) set of symbols. The Kahn domain $\Sigma^{\leq\omega}$
is the set of finite and infinite strings on $\Sigma$ together with
the pmetric
\[
p_{\Sigma^{\leq\omega}}\left(x,y\right):=2^{-\left(\text{the maximum length of common initial segments of }x\text{ and }y\right)}.
\]
The subspace $\Sigma^{<\omega}$ of finite strings is p-Cauchy incomplete.
The whole space $\Sigma^{\leq\omega}$ is p-Cauchy complete, and contains
$\Sigma^{<\omega}$ as a symmetrically dense subspace. Hence $\Sigma^{\leq\omega}$
is a p-Cauchy completion of $\Sigma^{<\omega}$. On the other hand,
the subspace $\Sigma^{<\omega}\setminus\set{\varepsilon}$, where
$\varepsilon$ is the empty string, is p-Cauchy incomplete. The subspace
$\Sigma^{\leq\omega}\setminus\set{\varepsilon}$ is p-Cauchy complete,
and contains $\Sigma^{\leq\omega}\setminus\set{\varepsilon}$ as a
symmetrically dense subspace. Hence $\Sigma^{\leq\omega}\setminus\set{\varepsilon}$
is a p-Cauchy completion of $\Sigma^{<\omega}\setminus\set{\varepsilon}$.
However, $\Sigma^{\leq\omega}$ contains $\Sigma^{<\omega}\setminus\set{\varepsilon}$
as a dense subspace. Therefore $\Sigma^{\leq\omega}$ is also a p-Cauchy
completion of $\Sigma^{<\omega}\setminus\set{\varepsilon}$. Notice
that $\Sigma^{<\omega}\setminus\set{\varepsilon}$ is not symmetrically
dense in $\Sigma^{\leq\omega}$. The empty string $\varepsilon$ plays
the same role as $b$ in the proof of \prettyref{lem:Asymmetric-superset}.
\end{example}

Symmetric denseness is a more appropriate notion of denseness for pmetric
spaces than topological denseness. First of all, the uniqueness theorem
holds under symmetric denseness \cite[Theorem 2]{GL15}, in contrast
with asymmetric (topological) denseness. Furthermore, the proof of
\cite[Theorem 2]{GL15} actually shows that the symmetric p-Cauchy
completion has the following universal property. (Note that the proof
of \cite[Proposition 1]{GL15} does not use the assumption that $Y$
is symmetrically dense in $Y^{\ast}$.)
\begin{thm}[{\cite[Proposition 1]{GL15}}]
Let $X$ be a pmetric space, $\bar{X}$ a p-Cauchy completion of
$X$, and $Y$ a p-Cauchy complete pmetric space. Suppose $X$ is
symmetrically dense in $\bar{X}$. Then every isometric embedding
$f\colon X\to Y$ has a unique extension $\bar{f}\colon\bar{X}\to Y$.
\end{thm}

\begin{thm}[{\cite[Theorem 2]{GL15}}]
Let $i\colon X\to\bar{X}$ and $j\colon X\to\tilde{X}$ be p-Cauchy
completions of a pmetric space $X$. Suppose $i\left(X\right)$ is
symmetrically dense in $\bar{X}$. Then there exists a unique isometric
embedding $f\colon\bar{X}\to\tilde{X}$ such that
\[
\xymatrix{ & X\ar[dl]_{i}\ar[dr]^{j}\\
\bar{X}\ar@{-->}[rr]^{f} &  & \tilde{X}
}
\]
is commutative.
\end{thm}

On the other hand, asymmetric p-Cauchy completions do not have such
properties.
\begin{thm}
Let $X$ be a non-empty pmetric space. There exist a p-Cauchy completion
$\tilde{X}$ of $X$, a p-Cauchy complete pmetric space $Y$, and
an isometric embedding $f\colon X\to Y$ such that $f$ has no isometric
extension $\tilde{f}\colon\tilde{X}\to Y$.
\end{thm}

\begin{proof}
Let $Y$ be an arbitrary p-Cauchy completion of $X$, and let $\tilde{X}$
be an asymmetric p-Cauchy completion of $Y$. Then $\tilde{X}$ is
also a p-Cauchy completion of $X$. Obviously the inclusion map $i_{XY}\colon X\to Y$
is an isometric embedding. However, $i_{XY}$ cannot be isometrically
extended to $\tilde{X}\to Y$.
\end{proof}
\begin{example}
Recall that $\Sigma^{\leq\omega}\setminus\set{\varepsilon}$ is p-Cauchy
complete and that $\Sigma^{\leq\omega}$ is an (asymmetric) p-Cauchy
completion of $\Sigma^{\leq\omega}\setminus\set{\varepsilon}$ (see
\prettyref{exa:Kahn-domain}). The identity map $\Sigma^{\leq\omega}\setminus\set{\varepsilon}\to\Sigma^{\leq\omega}\setminus\set{\varepsilon}$
is an isometric embedding having no isometric extension $\Sigma^{\leq\omega}\to\Sigma^{\leq\omega}\setminus\set{\varepsilon}$.
\end{example}

\section{\label{sec:Nonstandard-treatment}Nonstandard treatment of pmetric
completions}

It is well-known that the Cauchy completion of a metric space (and
of a uniform space) can be constructed via nonstandard analysis (see
Stroyan and Luxemburg \cite[8.4.28 Theorem]{SL76} or Hurd and Loeb
\cite[III.3.17 Theorem]{HR85}). We first summarise that construction.
Let $X$ be a metric space. Fix a sufficiently saturated nonstandard
extension $\ns{X}$ of $X$.\footnote{For our purpose, we only need to assume the $\aleph_{1}$-saturation
property. In fact, every proper elementary extension of a sufficiently
rich (set-theoretic) universe automatically satisfies the $\aleph_{1}$-saturation
condition.}
\begin{defn}[\cite{SL76}]
A point $x\in\ns{X}$ is said to be \emph{prenearstandard} if for
each $\varepsilon\in\mathbb{R}_{>0}$ there exists an $x_{\varepsilon}\in X$
such that $\ns{d_{X}}\left(x_{\varepsilon},x\right)<\varepsilon$,
i.e. $x\in\ns{B_{\varepsilon}}\left(x_{\varepsilon}\right)$. Denote
the set of all prenearstandard points by $\PNS\left(X\right)$.
\end{defn}

\begin{defn}[\cite{Rob66}]
Two points $x,y\in\ns{X}$ are said to be \emph{infinitesimally close}
(denoted by $x\approx_{X}y$) if the distance $\ns{d_{X}}\left(x,y\right)$
is infinitesimal.
\end{defn}

Then the quotient set 
\[
\bar{X}:=\PNS\left(X\right)/\approx_{X}
\]
equipped with the metric
\[
d_{\bar{X}}\left(\left[x\right],\left[y\right]\right):=\st{\left(\ns{d_{X}}\left(x,y\right)\right)}
\]
is a Cauchy completion of $X$, where $\st{x}$ refers to the standard
part of a finite hyperreal number $x$, which is a unique (standard)
real number infinitesimally close to $x$.

By mimicking the above construction, one can reconstruct the symmetric
p-Cauchy completion of a pmetric space, originally constructed by
Ge and Lin \cite{GL15}. Let $X$ be a pmetric space.
\begin{defn}
A point $x\in\ns{X}$ is said to be\emph{ prenearstandard} if for
each $\varepsilon\in\mathbb{R}_{>0}$ there exists an $x_{\varepsilon}\in X$
such that $\ns{p_{X}}\left(x_{\varepsilon},x\right)<p_{X}\left(x_{\varepsilon},x_{\varepsilon}\right)+\varepsilon$,
i.e. $x\in\ns{B_{\varepsilon}}\left(x_{\varepsilon}\right)$. Denote
the set of all prenearstandard points by $\PNS\left(X\right)$.
\end{defn}

Unlike the case of metric spaces, we need the following stronger property.
\begin{defn}
A point $x\in\ns{X}$ is said to be \emph{symmetrically prenearstandard}
if for each $\varepsilon\in\mathbb{R}_{>0}$ there exists an $x_{\varepsilon}\in X$
such that $\ns{p_{X}}\left(x_{\varepsilon},x\right)<p_{X}\left(x_{\varepsilon},x_{\varepsilon}\right)+\varepsilon$
and $\ns{p_{X}}\left(x,x_{\varepsilon}\right)<\ns{p_{X}}\left(x,x\right)+\varepsilon$,
i.e. $x\in\ns{B_{\varepsilon}}\left(x_{\varepsilon}\right)$ and $x_{\varepsilon}\in\ns{B_{\varepsilon}}\left(x\right)$.
Denote the set of all symmetrically prenearstandard points by $\SPNS\left(X\right)$.
\end{defn}

\begin{defn}
Two points $x,y\in\ns{X}$ are said to be \emph{infinitesimally close}
(denoted by $x\approx_{X}^{p}y$) if $\ns{p_{X}}\left(x,x\right)\approx_{\mathbb{R}}\ns{p_{X}}\left(x,y\right)\approx_{\mathbb{R}}\ns{p_{X}}\left(y,y\right)$.
\end{defn}

Note that if $X$ is a metric space, symmetric prenearstandardness
is equivalent to prenearstandardness, and $\approx_{X}^{p}$ coincides
with $\approx_{X}$.
\begin{claim}
$\approx_{X}^{p}$ is an equivalence relation on $\ns{X}$.
\end{claim}

\begin{proof}
The reflexivity and the symmetricity are trivial. Only the transitivity
is nontrivial. Let $x,y,z\in\ns{X}$ and suppose that $x\approx_{X}^{p}y\approx_{X}^{p}z$.
Then
\begin{align*}
\ns{p_{X}}\left(x,z\right) & \leq\ns{p_{X}}\left(x,y\right)+\ns{p_{X}}\left(y,z\right)-\ns{p_{X}}\left(y,y\right)\\
 & \approx_{\mathbb{R}}\ns{p_{X}}\left(y,y\right)+\ns{p_{X}}\left(z,z\right)-\ns{p_{X}}\left(y,y\right)\\
 & \approx_{\mathbb{R}}\ns{p_{X}}\left(z,z\right)\\
 & \leq\ns{p_{X}}\left(x,z\right).
\end{align*}
Hence $\ns{p_{X}}\left(x,z\right)\approx_{\mathbb{R}}\ns{p_{X}}\left(z,z\right)$.
Similarly $\ns{p_{X}}\left(x,z\right)\approx_{\mathbb{R}}\ns{p_{X}}\left(x,x\right)$.
Therefore $x\approx_{X}^{p}y$.
\end{proof}
Now consider the quotient set
\[
\bar{X}:=\SPNS\left(X\right)/\approx_{X}^{p}
\]
together with the pmetric
\[
p_{\bar{X}}\left(\left[x\right],\left[y\right]\right):=\st{\left(\ns{p_{X}}\left(x,y\right)\right)}.
\]
Let us show that $\left(\bar{X},p_{\bar{X}}\right)$ is an asymmetric
completion of $\left(X,p_{X}\right)$.
\begin{claim}
\label{claim:welldef-of-p-barX}$p_{\bar{X}}$ is well-defined.
\end{claim}

\begin{proof}
We must verify that $\ns{p_{X}}\left(x,y\right)$ is finite and invariant
under infinitesimal closeness.

Let $x,y\in\SPNS\left(X\right)$. Since $x,y$ are prenearstandard,
we can take $x_{1},y_{1}\in X$ such that $\ns{p_{X}}\left(x_{1},x\right)\leq p_{X}\left(x_{1},x_{1}\right)+1$
and $\ns{p_{X}}\left(y_{1},y\right)\leq p_{X}\left(y_{1},y_{1}\right)+1$.
Then
\begin{align*}
\ns{p_{X}}\left(x,y\right) & \leq\ns{p_{X}}\left(x,x_{1}\right)+p_{X}\left(x_{1},y_{1}\right)+\ns{p_{X}}\left(y_{1},y\right)-p_{X}\left(x_{1},x_{1}\right)-p_{X}\left(y_{1},y_{1}\right)\\
 & \leq p_{X}\left(x_{1},x_{1}\right)+1+p_{X}\left(x_{1},y_{1}\right)+p_{X}\left(y_{1},y_{1}\right)+1-p_{X}\left(x_{1},x_{1}\right)-p_{X}\left(y_{1},y_{1}\right)\\
 & =p_{X}\left(x_{1},y_{1}\right)+2.
\end{align*}
Hence $\ns{p_{X}}\left(x,y\right)$ is finite.

Let $x,y\in\SPNS\left(X\right)$ and suppose that $x\approx_{X}^{p}x'$
and $y\approx_{X}^{p}y'$. Then
\begin{align*}
\ns{p_{X}}\left(x,y\right) & \leq\ns{p_{X}}\left(x,x'\right)+\ns{p_{X}}\left(x',y'\right)+\ns{p_{X}}\left(y',y\right)-\ns{p_{X}}\left(x',x'\right)-\ns{p_{X}}\left(y',y'\right)\\
 & \approx_{\mathbb{R}}\ns{p_{X}}\left(x',x'\right)+\ns{p_{X}}\left(x',y'\right)+\ns{p_{X}}\left(y',y'\right)-\ns{p_{X}}\left(x',x'\right)-\ns{p_{X}}\left(y',y'\right)\\
 & =\ns{p_{X}}\left(x',y'\right),\\
\ns{p_{X}}\left(x',y'\right) & \leq\ns{p_{X}}\left(x',x\right)+\ns{p_{X}}\left(x,y\right)+\ns{p_{X}}\left(y,y'\right)-\ns{p_{X}}\left(x,x\right)-\ns{p_{X}}\left(y,y\right)\\
 & \approx_{\mathbb{R}}\ns{p_{X}}\left(x,x\right)+\ns{p_{X}}\left(x,y\right)+\ns{p_{X}}\left(y,y\right)-\ns{p_{X}}\left(x,x\right)-\ns{p_{X}}\left(y,y\right)\\
 & =\ns{p_{X}}\left(x,y\right),
\end{align*}
so $\ns{p_{X}}\left(x,y\right)\approx_{\mathbb{R}}\ns{p_{X}}\left(x',y'\right)$.
\end{proof}
\begin{claim}
\label{claim:p-is-a-pmetric}$p_{\bar{X}}$ is a pmetric on $\bar{X}$.
\end{claim}

\begin{proof}
$p_{\bar{X}}$ inherits the pmetric property from $p_{X}$ by the
transfer principle.
\begin{enumerate}
\item[(P1)] 
\begin{align*}
p_{\bar{X}}\left(\left[x\right],\left[x\right]\right)=p_{\bar{X}}\left(\left[x\right],\left[y\right]\right)=p_{\bar{X}}\left(\left[y\right],\left[y\right]\right) & \implies\ns{p_{X}}\left(x,x\right)\approx_{\mathbb{R}}\ns{p_{X}}\left(x,y\right)\approx_{\mathbb{R}}\ns{p_{X}}\left(y,y\right).\\
 & \implies x\approx_{X}^{p}y\\
 & \implies\left[x\right]=\left[y\right].
\end{align*}
\item[(P2)] 
\begin{align*}
p_{\bar{X}}\left(\left[x\right],\left[x\right]\right) & =\st{\left(\ns{p_{X}}\left(x,x\right)\right)}\\
 & \leq\st{\left(\ns{p_{X}}\left(x,y\right)\right)}\\
 & =p_{\bar{X}}\left(\left[x\right],\left[y\right]\right).
\end{align*}
\item[(P3)] 
\begin{align*}
p_{\bar{X}}\left(\left[x\right],\left[y\right]\right) & =\st{\left(\ns{p_{X}}\left(x,y\right)\right)}\\
 & =\st{\left(\ns{p_{X}}\left(y,x\right)\right)}\\
 & =p_{\bar{X}}\left(\left[y\right],\left[x\right]\right).
\end{align*}
\item[(P4)] 
\begin{align*}
p_{\bar{X}}\left(\left[x\right],\left[z\right]\right)+p_{\bar{X}}\left(\left[y\right],\left[y\right]\right) & =\st{\left(\ns{p_{X}}\left(x,z\right)\right)}+\st{\left(\ns{p_{X}}\left(y,y\right)\right)}\\
 & =\st{\left(\ns{p_{X}}\left(x,z\right)+\ns{p_{X}}\left(y,y\right)\right)}\\
 & \leq\st{\left(\ns{p_{X}}\left(x,y\right)+\ns{p_{X}}\left(y,z\right)\right)}\\
 & =\st{\left(\ns{p_{X}}\left(x,y\right)\right)}+\st{\left(\ns{p_{X}}\left(y,z\right)\right)}\\
 & =p_{\bar{X}}\left(\left[x\right],\left[y\right]\right)+p_{\bar{X}}\left(\left[y\right],\left[z\right]\right).\qedhere
\end{align*}
\end{enumerate}
\end{proof}
\begin{claim}
The map $i\colon X\ni x\mapsto\left[x\right]\in\bar{X}$ is an isometric
embedding.
\end{claim}

\begin{proof}
Immediate.
\end{proof}
\begin{claim}
\label{claim:X-is-symmetrically-dense}$i\left(X\right)$ is symmetrically
dense in $\bar{X}$.
\end{claim}

\begin{proof}
Let $\left[x\right]\in\bar{X}$. Since $x$ is symmetrically prenearstandard,
we can find, for each $\varepsilon>0$, an $x_{\varepsilon}\in X$
such that $\ns{p_{X}}\left(x_{\varepsilon},x\right)<\ns{p_{X}}\left(x_{\varepsilon},x_{\varepsilon}\right)+\varepsilon/2$
and $\ns{p_{X}}\left(x,x_{\varepsilon}\right)<\ns{p_{X}}\left(x,x\right)+\varepsilon/2$.
Then
\begin{align*}
p_{\bar{X}}\left(\left[x_{\varepsilon}\right],\left[x\right]\right) & =\st{\left(\ns{p_{X}}\left(x_{\varepsilon},x\right)\right)}\\
 & \leq\st{\left(\ns{p_{X}}\left(x_{\varepsilon},x_{\varepsilon}\right)+\varepsilon/2\right)}\\
 & <\st{\left(\ns{p_{X}}\left(x_{\varepsilon},x_{\varepsilon}\right)\right)}+\varepsilon\\
 & =p_{\bar{X}}\left(\left[x_{\varepsilon}\right],\left[x_{\varepsilon}\right]\right)+\varepsilon,
\end{align*}
so $p_{\bar{X}}\left(\left[x_{\varepsilon}\right],\left[x\right]\right)<p_{\bar{X}}\left(\left[x_{\varepsilon}\right],\left[x_{\varepsilon}\right]\right)+\varepsilon$.
Similarly we have that $p_{\bar{X}}\left(\left[x\right],\left[x_{\varepsilon}\right]\right)<p_{\bar{X}}\left(\left[x\right],\left[x\right]\right)+\varepsilon$.
Hence $i\left(X\right)$ is symmetrically dense in $\bar{X}$.
\end{proof}
\begin{claim}
$\bar{X}$ is p-Cauchy complete.
\end{claim}

\begin{proof}
According to \cite[Lemma 4]{GL15}, it suffices to show that every
p-Cauchy sequence in $i\left(X\right)$ p-converges in $\bar{X}$.
Let $\set{x_{n}}_{n\in\mathbb{N}}$ be a p-Cauchy sequence in $X$.
Fix an infinite $\omega\in\ns{\mathbb{N}}\setminus\mathbb{N}$. Let
us verify that $\set{\left[x_{n}\right]}_{n\in\mathbb{N}}$ p-converges
to $\left[\ns{x_{\omega}}\right]$.

Let $r:=\lim_{n,m\to\infty}p_{X}\left(x_{n},x_{m}\right)$. Then for
each $\varepsilon>0$ there exists an $N_{\varepsilon}\in\mathbb{N}$
such that $\left|p_{X}\left(x_{n},x_{m}\right)-r\right|\leq\varepsilon$
for any $n,m\geq N_{\varepsilon}$. By the transfer principle, $\left|\ns{p_{X}}\left(\ns{x_{n}},\ns{x_{\omega}}\right)-r\right|\leq\varepsilon$
holds for all $n\geq N_{\varepsilon}$, where $n$ can be nonstandard.
Since $\ns{p_{X}}\left(x_{N_{\varepsilon}},\ns{x_{\omega}}\right)\leq r+\varepsilon$
and $r-\varepsilon\leq p_{X}\left(x_{N_{\varepsilon}},x_{N_{\varepsilon}}\right)$,
we have that $\ns{p_{X}}\left(x_{N_{\varepsilon}},\ns{x_{\omega}}\right)\leq r+\varepsilon\leq p_{X}\left(x_{N_{\varepsilon}},x_{N_{\varepsilon}}\right)+2\varepsilon$.
Similarly, since $r-\varepsilon\leq\ns{p_{X}}\left(\ns{x_{\omega}},\ns{x_{\omega}}\right)$,
it follows that $\ns{p_{X}}\left(\ns{x_{\omega}},x_{N_{\varepsilon}}\right)\leq r+\varepsilon\leq\ns{p_{X}}\left(\ns{x_{\omega}},\ns{x_{\omega}}\right)+2\varepsilon$.
Hence $\ns{x_{\omega}}$ is symmetrically prenearstandard, i.e. $\left[\ns{x_{\omega}}\right]\in\bar{X}$.

Since $r-\varepsilon\leq\ns{p_{X}}\left(x_{n},\ns{x_{\omega}}\right),p_{X}\left(x_{n},x_{m}\right),\ns{p_{X}}\left(\ns{x_{\omega}},\ns{x_{\omega}}\right)\leq r+\varepsilon$
for all $n,m\geq\mathbb{N}_{\varepsilon}$ and $\varepsilon\in\mathbb{R}_{>0}$,
it follows that $\lim_{n\to\infty}p_{\bar{X}}\left(\left[x_{n}\right],\left[\ns{x_{\omega}}\right]\right)=\lim_{n,m\to\infty}p_{\bar{X}}\left(\left[x_{n}\right],\left[x_{m}\right]\right)=p_{\bar{X}}\left(\left[\ns{x_{\omega}}\right],\left[\ns{x_{\omega}}\right]\right)=r$.
Therefore $\set{\left[x_{n}\right]}_{n\in\mathbb{N}}$ p-converges
to $\left[\ns{x_{\omega}}\right]$.
\end{proof}
Thus the proof is completed. Notice that \prettyref{claim:X-is-symmetrically-dense}
depends on \emph{symmetric} prenearstandardness, while \prettyref{claim:p-is-a-pmetric}
does not. So one can obtain another pmetric space by replacing $\SPNS\left(X\right)$
with $\PNS\left(X\right)$. The resulting space $\PNS\left(X\right)/\approx_{X}^{p}$
has $X$ as a dense subspace, although may not have it as a symmetrically
dense subspace. The study of this space is a future work.

Our nonstandard construction might be useful for integrating various
completions of pmetric spaces. For example, $0$-Cauchy completions
developed by Moshokoa \cite{Mos16} can be realised as subspaces
of the above constructed completions:
\begin{align*}
\bar{X}_{0} & :=\set{\left[x\right]|x\in X}\cup\set{\left[x\right]|x\in\SPNS\left(X\right)\text{ and }\ns{p_{X}}\left(x,x\right)\approx_{\mathbb{R}}0}\\
 & \cong X\cup\set{x\in\SPNS\left(X\right)|\ns{p_{X}}\left(x,x\right)\approx_{\mathbb{R}}0}/\approx_{X}^{p}
\end{align*}
is $0$-Cauchy complete, and contains $X$ as a symmetrically dense
subspace.

\bibliographystyle{amsplain}
\bibliography{bibliography}

\section*{Corrigendum}
In the accepted version (this version), there was a typo in the sentence before \prettyref{claim:welldef-of-p-barX}. The space $\left(\bar{X}, p_{\bar{X}}\right)$ is a \emph{symmetric} p-Cauchy completion of $\left(X, p_{X}\right)$, not an \emph{asymmetric} one. This has been corrected in the final version.

Regrettably, in the the final version, there was a typo in the last sentence of \prettyref{sec:Nonstandard-treatment}. The subspace $X$ of $\bar{X}_{0}$ is \emph{symmetrically} dense, not \emph{asymmetrically}. Also, ``saturation'' in footnote 1 should be replaced with ``enlarging'', a weaker form of saturation.

These errors do not affect the conclusions of the article. The author apologises for these errors.

\end{document}